\documentclass[12pt]{amsart}%

\usepackage{amssymb}
\usepackage{amsmath}
\usepackage{amsthm}
\usepackage{amscd,mdwlist}
\usepackage[margin=1in]{geometry}%
\usepackage{hyperref}

\theoremstyle{plain}

\newtheorem{lemma}{Lemma}[section]

\theoremstyle{remark}
\newtheorem{definition}{Definition}[section]
\newtheorem{remark}{Remark}[section]
\newtheorem{examples}{Examples}[section]
\newtheorem{assumption}{Assumption}[section]

\DeclareMathOperator{\supp}{supp}

\DeclareMathOperator{\clos}{clos}
\DeclareMathOperator{\lin}{span}
\DeclareMathOperator{\cpct}{Cap}

\DeclareMathOperator{\dom}{dom}
\DeclareMathOperator{\diverg}{div}
\DeclareMathOperator{\grad}{grad}
\DeclareMathOperator{\im}{Im}

\DeclareMathOperator{\tr}{tr}

\begin{document}

\title{Finite energy coordinates and vector analysis on fractals}
\author{Michael Hinz$^1$}
\address{$^1$ Fakult\"at f\"ur Mathematik, Universit\"at Bielefeld, Postfach 100131, 33501 Bielefeld, Germany}
\email{mhinz@math.uni-bielefeld.de}

\author{Alexander Teplyaev$^2$}
\address{$^2$Department of Mathematics, University of Connecticut, Storrs, CT 06269-3009 USA}
\email{Alexander.Teplyaev@uconn.edu}
\thanks{$^2$Research supported in part by NSF grant DMS-0505622}

\date{\today}

\begin{abstract}
We consider (locally) energy finite coordinates associated with a strongly local regular Dirichlet form on a metric measure space. We give coordinate formulas for substitutes of tangent spaces, for gradient and divergence operators and for the infinitesimal generator. As examples we discuss Euclidean spaces, Riemannian local charts, domains on the Heisenberg group and the measurable Riemannian geometry on the Sierpinski gasket.\tableofcontents\end{abstract}\maketitle

\section{Introduction}

Suitable coordinate maps are 
tools in many branches of geometry. For instance, 
smooth coordinate changes are the crucial ingredient in the definition of a differentiable structure on a manifold and therefore omnipresent in 
differential geometry (e.g. \cite{GHL90, Jost}). 
For a general metric measure space we can 
not expect to find local coordinates that transform smoothly. 
However, in the field of analysis on fractals Kusuoka \cite{Ku89, Ku93}, Kigami \cite{Ki93, Ki08}, Strichartz \cite{St00}, Teplyaev \cite{T00}, Hino \cite{Hino13}, Kajino \cite{Ka12,Ka12t} and others have contributed to a concept that is now referred to as 'measurable Riemannian geometry'. 
This concept is based on Dirichlet forms and involves the use of harmonic functions as 
'global coordinates'. 
In probability similar ideas can already be found in works of 
Doob, Dynkin, Skorohod. On the other hand there is 
recent progress in the studies of a first order calculus on fractals, \cite{CGIS,CS09, HRT13, HTc, HTams, IRT} again based on Dirichlet form theory, \cite{ChFu12, FOT94}, which allows to discuss differential $1$-forms 
and
vector fields, partially based on  \cite{CS03, S90}.

In the present note we consider metric measure spaces, 
equipped with a strongly local Dirichlet form and consider associated (locally) energy finite coordinates. 
Analogously to Riemannian geometry, we 
provide coordinate expressions for the gradient and divergence operators (derivation and coderivation) used in the first order theory, and for the infinitesimal generator. 
The purpose of the present paper is rather didactical and we do do not claim any substantial novelty. Instead we hope to facilitate understanding of how the measurable first order calculus is related to Euclidean, Riemannian, sub-Riemannian and measurable Riemannian situations.

\section{Preliminaries}
Let $X$ be a locally compact separable metric space and $\mu$ a nonnegative Radon measure on $X$ such that $\mu(U)>0$ for all nonempty open $U\subset X$. Let $(\mathcal{E},\mathcal{F})$ be a \emph{strongly local regular Dirichlet form} on $L_2(X,\mu)$, that is: 
\,{\bf(1)}~$\mathcal{F}$ is a dense subspace of $L_2(X,\mu)$ and $\mathcal{E}:\mathcal{F}\times\mathcal{F}\to\mathbb{R}$ is a nonnegative definite symmetric bilinear form; we denote  $\mathcal{E}(f):=\mathcal{E}(f,f)$; 
\,{\bf(2)}~
$\mathcal{F}$ is a Hilbert space 
with the norm $\sqrt{\mathcal{E}_1(f)}:=\big({\mathcal{E}(f)+\left\| f \right\|^2_{L_2(X,\mu)}}\big)^{1/2}$\,;
\,{\bf(3)}~for any $f\in\mathcal{F}$ we have $(f\vee 0)\wedge 1\in\mathcal{F}$ and $\mathcal{E}((f\vee 0)\wedge 1)\leq \mathcal{E}(f)$,
where $f\vee g:=\max\left\lbrace f,g\right\rbrace$ and $f\wedge g:=\min\left\lbrace f,g\right\rbrace$;
\,{\bf(4)}~
$\mathcal{C}:=\mathcal{F}\cap C_c(X)$  
is dense both in $\mathcal{F}$ with respect to the norm $\sqrt{\mathcal{E}_1(f)}$,
and in the space $C_c(X)$ of continuous compactly supported functions with respect to the uniform norm;
\,{\bf(5)}~if $f,g\in\mathcal{C}$ are such that $g$ is constant on an open neighborhood of $\supp f$ then $\mathcal{E}(f,g)=0$.

\medskip

To each Dirichlet form $(\mathcal{E},\mathcal{F})$ on $L_2(X,\mu)$ there exists a unique non-positive self-adjoint operator $(L, \dom\:L)$, called the \emph{infinitesimal generator of $(\mathcal{E},\mathcal{F})$}, such that $\dom\:L\subset \mathcal{F}$ and 
\[\mathcal{E}(f,g)=-\left\langle Lf, g\right\rangle_{L_2(X,\mu)}\] 
for all $f\in \dom\:L$ and $g\in\mathcal{F}$. See \cite{ChFu12, FOT94}. For pointwise products we have
\begin{equation}\label{E:BDbound}
\mathcal{E}(fg)^{1/2}\leq \mathcal{E}(f)^{1/2}\left\|g\right\|_{L_\infty(X,\mu)}+\mathcal{E}(g)^{1/2}\left\|f\right\|_{L_\infty(X,\mu)}, \ \ f,g\in \mathcal{F}\cap L_\infty(X,\mu),
\end{equation}
\cite[Corollary I.3.3.2]{BH91}, and in particular, the space $\mathcal{C}:=\mathcal{F}\cap C_c(X)$ is an algebra. For any $f,g\in \mathcal{C}$ a signed Radon measure $\Gamma(f,g)$ on $X$ is defined by 
\begin{equation}\label{E:energymeas}
\int_X\varphi d\Gamma(f,g)=\frac12\big(\mathcal{E}(f,g\varphi)+\mathcal{E}(g,f\varphi)-\mathcal{E}(fg,\varphi)\big), \ \ \varphi\in\mathcal{C}.
\end{equation}
 By approximation in $\mathcal{F}$ we also define $\Gamma(f,g)$ for any   $f,g\in\mathcal{F}$, referred to as
the \emph{(mutual) energy measure of $f$ and $g$}, see \cite{FOT94}. 
We denote the nonnegative measure $\Gamma(f)=\Gamma(f,f)$. 
Below it will be advantageous to consider functions that are only locally of finite energy. 
%
We define $\mathcal{F}_{loc}$ as the set of functions $ f\in L_{2,loc}(X,\mu) $ such that for any relatively compact open set $V\subset X$ there exists some $u\in\mathcal{F}$ such that $f|_V=u|_V$ $\mu$-a.e. 
Using an exhaustion of $X$ by an increasing sequence of relatively compact open sets and related cut-off functions we can define $\Gamma(f)$ for $f\in\mathcal{F}_{loc}$.
If $V$ is relatively compact open and $u\in \mathcal{F}$ agrees with $f$ $\mu$-a.e. on $V$ then 
\begin{equation}\label{E:GammaV}
\Gamma(f)|_V=\Gamma(u)|_V.
\end{equation}

\begin{examples}
A prototype for a strongly local regular Dirichlet form is the \emph{Dirichlet integral}
\[\mathcal{E}(f)=\int_{\mathbb{R}^n}|\nabla f|^2\:dx\]
on $L_2(\mathbb{R}^n)$, where $\mathcal{F}$ is the Sobolev space $H^1(\mathbb{R}^n)$ of functions $f\in L_2(\mathbb{R}^n)$ with $\frac{\partial f}{\partial x_i}\in L_2(\mathbb{R}^n)$ for all $i$. Note that $C^1_c(\mathbb{R}^n)$ is dense in $H^1(\mathbb{R}^n)$ and in $C_c(\mathbb{R}^n)$. The generator is the Laplacian $L=\Delta$ and the energy measures are given by $\Gamma(f)=|\nabla f|^2\:dx$.
\end{examples}

A nonnegative Radon measure $m$ on $X$ is called \emph{energy dominant} if all energy measures $\Gamma(f)$, $f\in \mathcal{F}$, are absolutely continuous with respect to $m$, \cite{Hino13, HRT13}. By $\frac{d\Gamma(f)}{dm}$ we denote the corresponding Radon-Nikodym densities. We say that a nonnegative Radon measure $m$ on $X$ is \emph{smooth} if it assigns zero to any Borel set of zero capacity. The capacity $\cpct(A)$ of an open set $A$ is defined as 
\[\cpct(A):=\inf_u\inf_B\left\lbrace \mathcal{E}_1(u): \text{$u\in\mathcal{F}$ and $u\geq 1$ $\mu$-a.e. on $B$, \text{$B\supset A$, $B$ open}}\right\rbrace,\] 
and if the infimum is taken over the empty set, $\cpct(A)$ is defined to be infinity. 
For further details see \cite{ChFu12, FOT94}. The reference measure $\mu$ is always smooth.

Let $\varphi\in \mathcal{C}$, let $V$ be a relatively compact open neighborhood $V$ of $\supp \varphi$ and suppose $(f_n)_n\subset \mathcal{F}_{loc}$. We say that \emph{$\varphi$ is locally approximated by the sequence $(f_n)_n$ on $V$} if there is a sequence $(u_n)_n\subset \mathcal{F}$ with $\lim_n \mathcal{E}_1(\varphi-u_n)=0$ and $f_n|_V=u_n|_V$ $\mu$-a.e.
for all $n$. 
The following lemma follows from (\ref{E:GammaV}), \cite[Theorem 2.1.4]{FOT94} and the smoothness of $m$.

\begin{lemma}\label{L:locapprox}
Let $\varphi\in\mathcal{C}$ and let $V$ be a relatively compact open neighborhood $V$ of $\supp \varphi$. Suppose that $\varphi$ is locally approximated by $(f_n)_n\subset \mathcal{F}_{loc}$ on $V$. Then 
$\lim_n \Gamma(\varphi-f_n)(V)=0$. If, moreover, $m$ is a smooth energy dominant measure and the functions $f_n$ are continuous on $V$, then there is a subsequence $(f_{n_k})_k$ such that $\lim_k f_{n_k}=\varphi$ $m$-a.e. on $V$.
\end{lemma}

\section{Finite energy coordinates}

Let $y=(y^i)_{i\in I}$ be a finite or countable collection of locally bounded functions $y^i$. Given a finite ordered subset $J=( n_1,\dots, n_k) $ of $I$, the space of all functions of form $f=F(y^{n_1},...,y^{n_k})$,
where the functions $F$ are polynomials in $k$ variables and such that $F(0)=0$, will be denoted by $\mathcal{P}_J(y)$. For a fixed collection $(y^i)_{i\in I}$ set
\begin{equation}\label{E:union}
\mathcal{P}(y):=\bigcup_{J\subset I} \mathcal{P}_J(y),
\end{equation}
the union taken over all ordered finite subsets $J$ of $I$. Note that $\mathcal{P}(y)$ is an algebra of locally bounded functions. For any $k$ we regard the space $\mathbb{R}^k$ as a subspace of $\mathbb{R}^\mathbb{N}$ by putting $(v_1,v_2,\dots, v_k, v_{k+1}, v_{k+2}, \dots):=(v_1,v_2,\dots, v_k, 0, 0,\dots)$ for any given $(v_1,v_2,\dots, v_k)\in\mathbb{R}^k$. Similarly, we consider $(k\times k)$-matrices as linear operators from $\mathbb{R}^\mathbb{N}$ to $\mathbb{R}^\mathbb{N}$.

\begin{definition}\label{D:coordseq}
Let $m$ be a smooth energy dominant measure for $(\mathcal{E},\mathcal{F})$. A finite or countable collection $y=(y^i)_{i\in I}$ of continuous and locally bounded functions $y^i\in \mathcal{F}_{loc}$ is called a \emph{coordinate sequence for $(\mathcal{E},\mathcal{F})$ with respect to $m$} if
\smallskip\item[\ (i)] any $\varphi\in\mathcal{C}$ can locally be approximated on a relatively compact neighborhood $V$ of $\supp \varphi$ by a sequence of 
elements of $\mathcal{P}(y)$ 
\item[\ (ii)] for any $i\in I$ we have
\[\frac{d\Gamma(y^i)}{dm}\in L_1(X,m)\cap L_\infty(X,m),\]
and for any $i$ and $j$
\[Z^{ij}(x):=\frac{d\Gamma(y^i,y^j)}{dm}(x)\]
are Borel functions (versions) such that for $m$-a.e. $x\in X$, 
$
Z(x):=
(Z^{ij}(x))_{ij=1}^\infty$ defines a bounded 
 symmetric  nonnegative definite
linear operator $Z(x):l_2 \to l_2$.
\item[\ (iii)] We say that \emph{the coordinates $y^i$ have finite energy} if $y^i\in\mathcal{F}$ for all $i\in I$.
\end{definition}

A coordinate sequence $y=(y^i)_{i\in I}$ induces a mapping $y:X\to \mathbb{R}^\mathbb{N}$.

\begin{remark}
Condition (i) in Definition \ref{D:coordseq} makes sense because we have $\mathcal{P}(y)\subset\mathcal{F}_{loc}$. If the coordinates $y^i$ have finite energy the inclusion $\mathcal{P}(y)\subset\mathcal{F}$ is clear from (\ref{E:BDbound}). To see this inclusion in the general case it suffices to show that for any continuous and locally bounded $f,g\in\mathcal{F}_{loc}$ we have $fg\in 
\mathcal{F}_{loc}$. Clearly $fg\in L_{2,loc}(X,\mu)$. Further, given a relatively compact open set $V\subset X$ we can find a suitable cutoff function $\chi\in\mathcal{C}$ with $0\leq \chi\leq 1$ and $\chi\equiv 1$ on $V$, a relatively compact open neighborhood of $\supp\chi$ and functions $u,v\in\mathcal{F}$ such that $f|_U=u|_U$ and $g|_U=v|_U$ $\mu$-a.e. Clearly $\chi u\in L_2(X,\mu)$, and using locality, \cite[Corollary 3.2.1]{FOT94}, 
\[\mathcal{E}(\chi u)^{1/2}=\left(\int_U d\Gamma(\chi u)\right)^{1/2}\leq \left(\int_U \chi^2d\Gamma(u)\right)^{1/2}+\left(\int_U \widetilde{u}^2d\Gamma(\chi)\right)^{1/2},\]
where $\widetilde{u}$ is a quasi-continuous version of $u$. See e.g. \cite[Chapter II]{FOT94} for quasi-continuity and the Appendix in \cite{HKT13} for comments on the formula (which also follows from Cauchy-Schwarz applied to (\ref{E:HSP}) below). Approximating $u$ in $\mathcal{E}_1^{1/2}$-norm by a sequence from $\mathcal{C}$ we see that $\chi u$ is the limit in $\mathcal{E}_1^{1/2}$-norm of a sequence from $\mathcal{C}$, and by completeness $\chi u$ is in $\mathcal{F}$. Similarly for $\chi v$. Both functions are bounded on $U$ $\mu$-a.e. and vanish outside $U$, hence are also members of $L_2(X,\mu)$. Therefore $\chi^2 u v\in\mathcal{F}$ by (\ref{E:BDbound}), what implies $fg\in \mathcal{F}_{loc}$. 
\end{remark}

In Section \ref{S:construct} we show that (under an additional continuity assumption) it is always possible to construct a smooth finite energy dominant measure and a corresponding coordinate sequence of energy finite coordinates. The following examples relate Definition \ref{D:coordseq} to well known situations.

\begin{examples}\label{Ex:1}
\ (1) Consider
\[\mathcal{E}(f):=\sum_{i,j=1}^n \int_{\mathbb{R}^n} a_{ij}(x)\frac{\partial f}{\partial x_i}(x)\frac{\partial f}{\partial x_j}(x)dx,\ \ f\in C_c^1(\mathbb{R}^n),\]
where $a_{ij}=a_{ji}$ are bounded Borel functions satisfying $\sum_{i,j=1}^n a_{ij}(x)\xi_i\xi_j\geq c|\xi|^2$ with a universal constant $c>0$ for any $\xi\in\mathbb{R}^n$ and $\lambda^n$-a.e. $x\in \mathbb{R}^n$. Here $\lambda^n$ denotes the $n$-dimensional Lebesgue measure $\lambda^n(dx)=dx$. Then $(\mathcal{E}, C_c^1(\mathbb{R}^n))$ is closable in the space $L_2(\mathbb{R}^n)$, and its closure  $(\mathcal{E}, H^1(\mathbb{R}^n))$ is a strongly local regular Dirichlet form. Obviously $\lambda^n$ is smooth and energy dominant for $(\mathcal{E},H^1(\mathbb{R}^n))$. The Euclidean coordinates $y^k(x)=x_k$, $k=1,\dots, n$, form a coordinate sequence for $(\mathcal{E},H^1(\mathbb{R}^n))$ with respect to $\lambda^n$. Note that $\nabla y^k=e_k$ is the $k$-th unit vector in $\mathbb{R}^n$, and we have 
\[Z^{ij}(x)=a_{ij}(x)\ \ \text{ for $\lambda^n$-a.e. $x\in \mathbb{R}^n$}\] 
and $i,j=1,\dots n$. This shows (ii). If $\varphi\in C_c^1(\mathbb{R}^n)$ then we can find a relatively compact open set $V$ containing $\supp \varphi$ on which the function $\varphi$ can be approximated it in $C^1$-norm by a sequence of polynomials in the variables $x_1,\dots, x_n$, hence in the coordinates $y^1,\dots, y^n$. Multiplying these polynomials by a (nonnegative) $C^1$-cut-off function supported in $V$ and equal to one on $\supp\varphi$, the approximation is seen to take place in $H^1(\mathbb{R}^n)$.  As $C_c^1(\mathbb{R}^n)$ is dense in $H^1(\mathbb{R}^n)$, this implies (i). The coordinates $y^k$ do not have finite energy. 
\smallskip\item[\ (2)] Let $(M,g)$ be an $n$-dimensional Riemannian manifold, \cite{GHL90, Jost}, let $(V,y)$ be a local chart with coordinates $y=(y^1,\dots, y^n)$ and $U$ a relatively compact open set with $\overline{U}\subset V$. By $dvol$ we denote the Riemannian volume (restricted to $U$). The closure $(\mathcal{E},\mathring{H}^1(U))$ in $L_2(U,dvol)$ of 
\[\mathcal{E}(f):=\int_U\left\langle \nabla f, \nabla f\right\rangle_{T_xM}dvol(x),\ \ f\in C^1_c(U).\]
is a strongly local Dirichlet form. The reference measure $dvol$ is smooth and energy dominant, for any $k=1,\dots, n$ we have
\[\nabla y^k=g^{kj}\frac{\partial}{\partial y^j} \]
and
\[Z^{kk}=\left\langle \nabla y^k, \nabla y^k\right\rangle_{TM}=g^{kj}g^{ki}\left\langle \frac{\partial}{\partial y^i}, \frac{\partial}{\partial y^j}\right\rangle_{TM}=g^{kk}\]
and therefore (ii). Recall that 
\[\nabla f= g^{ij}\frac{\partial f}{\partial y^i}\frac{\partial}{\partial y^j}, \ \ g_{ij}=\left\langle\frac{\partial}{\partial y^i}, \frac{\partial}{\partial y^j}\right\rangle_{TM}\]
and $g^{ki}g_{ij}=\delta^k_j$. For a function $f\in C^1_c(U)$ the function $f\circ y^{-1}$ is a member of $C^1(W)$, and accordingly it can be approximated in $C^1(W)$-norm by a sequence $(p_m)_m$ of polynomials in the variables $y^1,\dots, y^n$. Consequently the functions $p_m\circ y$ approximate $f$ in $C^1(U)$-norm (note that the differentials $d((p_m\circ y)\circ y^{-1})(y(x))$ approximate $d(f\circ y^{-1})(y(x))$ uniformly in $x\in U$). This implies (i). Here the $y^i$ are not in $\mathcal{F}$ because they do not satisfy the Dirichlet boundary conditions on $\partial U$.
\smallskip\item[\ (3)] A sub-Riemannian example is given by the Heisenberg group $\mathbb{H}$, \cite{DM05, GL14, M02, Str86}, realized as $\mathbb{R}^3$ together with the non-commutative
multiplication
\[(\xi_1,\eta_1,\zeta_1)\cdot(\xi_2, \eta_2, \zeta_2):=(\xi_1+\xi_2, \eta_1+\eta_2, \zeta_1+\zeta_2+\xi_1\eta_2-\eta_1\xi_2).\]
Left multiplication by $(\xi,0,0)$ and $(0,\eta,0)$ yields the left-invariant vector fields 
\[X(q):=\frac{\partial}{\partial \xi}\Big|_q-\frac12\eta\frac{\partial}{\partial\zeta}\Big|_q \ \ \text{ and } \ \ Y(q):=\frac{\partial}{\partial \eta}\Big|_q+\frac12\xi\frac{\partial}{\partial\zeta}\Big|_q,\] 
and at each $q=(\xi,\eta,\zeta)\in\mathbb{H}$ the tangent vectors $X(q)$ and $Y(q)$ span a two-dimensional subspace $V_q$ of the tangent space $T_q\mathbb{H}\cong \mathbb{R}^3$. The sub-Riemannian metric is given by the inner products
$\left\langle\cdot,\cdot\right\rangle_{V_q}$ on the spaces $V_q$ that makes $(X(q),Y(q))$ an orthonormal basis, respectively. We use the Haar measure on $\mathbb{H}$, which coincides with the Lebesgue measure $\lambda^3$ on $\mathbb{R}^3$. Now let $U\subset \mathbb{H}$ be a connected bounded open set and consider the bilinear form 
\[\mathcal{E}(f):=\int_U((Xf)^2+(Yf)^2)\:d\lambda^3,\ \ f\in C_c^1(U).\]
Let $(\mathcal{E},\mathring{S}^1(U))$ denote the closure of $(\mathcal{E}, C_c^1(U))$ in $L_2(U)$. Obviously $\lambda^3$ is smooth and energy dominant. A coordinate sequence for $(\mathcal{E},\mathring{S}^1(U))$ and $\lambda^3$ is given by $y=(y^1,y^2,y^3):=(\xi,\eta,\zeta)$. Condition (i) follows again by polynomial approximation in $C_c^1(U)$. It is immediate that $Xy^1=1$, $Yy^1=0$, similarly for $y^2$, and $Xy^3=-\frac{\eta}{2}$, 
$Yy^3=\frac{\xi}{2}$, which yields the symmetric and nonnegative definite matrices 
\[Z(q)=\left(\begin{array}{ccc} 1 &0 &-\frac{\eta}{2}\\
0 & 1 & \frac{\xi}{2}\\
-\frac{\eta}{2} & \frac{\xi}{2} & \frac{\xi^2+\eta^2}{4}\end{array}\right),\]
so that (ii) is satisfied. For any $q\in\mathbb{H}$ the matrix $Z(q)$ has rank two. As in (2) the coordinates are not in $\mathcal{F}$.
\smallskip\item[\ (4)] We consider a prototype of a finitely ramified fractal in finite energy coordinates. Let $K$ denote the Sierpinski gasket, seen as the post-critically self-similar structure generated by the maps $f_j:\mathbb{R}^2\to\mathbb{R}^2$, $f_j(x)=\frac12(x+p_j)$,  $j=1,2,3$,
where $p_1$, $p_2$ and $p_3$ are the vertices of an equilateral triangle in $\mathbb{R}^2$. Let $(\mathcal{E},\mathcal{F})$ be the standard resistance form on $K$, obtained as the rescaled limit of discrete energy forms along a sequence of graphs with increasing vertex sets $V_n$ 'approximating $K$',
\[\mathcal{E}(f)=\lim_{n\to\infty}\left(\frac53\right)^n\sum_{p,q\in V_n}(f(p)-f(q))^2,\]
see e.g. \cite{Ki93, Ki01, Ku89, Ku93, Str06} for details. With $\left\lbrace p_1, p_2, p_3\right\rbrace$ as boundary and with Dirichlet boundary conditions there exist two 
harmonic functions $y^1, y^2\in\mathcal{F}$ with $\mathcal{E}(y^1)=\mathcal{E}(y^2)=1$ and $\mathcal{E}(y^1,y^2)=0$ such that the mapping $y:K\to\mathbb{R}^2$
\begin{equation}\label{E:SGy}
y(x):=(y^1(x), y^2(x)), \ \ x\in K,
\end{equation}
is a homeomorphism from $K$ onto its image $y(K)\subset \mathbb{R}^2$. We consider $K$ endowed with the \emph{Kusuoka measure} $\nu$, defined as the sum 
\[\nu:=\Gamma(y^1)+\Gamma(y^2)\]
of the energy measure $\Gamma(y^1)$ and $\Gamma(y^2)$ of $y^1$ and $y^2$, respectively. The resistance form $(\mathcal{E},\mathcal{F})$ induces a strongly local Dirichlet form on $L_2(K,\nu)$, for which the finite measure $\nu$ is smooth and energy dominant. The pair $(y^1, y^2)$ is a coordinate sequence for this form: Condition (ii) is satisfied by construction, condition (i) follows by polynomial approximation and the density of functions of type $F\circ y$, $F\in C^1(\mathbb{R}^2)$, in $\mathcal{F}$, see e.g. \cite{Ki93, Ku89, Ku93, T08}. The operators $Z(x)$ may be viewed as $(2\times 2)$-matrices, and for $\nu$-a.e. $x\in K$ the matrix $Z(x)$ is symmetric, nonnegative definite and has rank one. 
\end{examples}

\section{Energy, fibers and bundles}\label{S:fibers}

In what follows we will assume throughout that $(\mathcal{E},\mathcal{F})$ is a strongly local regular Dirichlet form on $L_2(X,\mu)$, $m$ is a smooth energy dominant measure and $y=(y^i)_{i\in I}$ be a coordinate sequence for $(\mathcal{E},\mathcal{F})$ with respect to $m$. 

We would like to emphasize that unless stated otherwise we do not assume that the reference measure itself is energy dominant or that the form $(\mathcal{E},\mathcal{F})$ has a restriction that is closable with respect to the energy dominant measure $m$ under consideration.

In Example \ref{Ex:1} (4), a
  well known formula of Kusuoka \cite{Ku89} and Kigami \cite{Ki93} is 
\begin{equation}\label{E:KuKi0}
\mathcal{E}(f,g)=\int_K\left\langle \nabla F(y), Z(x)\nabla G(y)\right\rangle_{\mathbb{R}^2} \nu(dx),
\end{equation}
for all $f=F\circ y$ and $g=G\circ y$ with $F,G\in C^1(\mathbb{R}^2)$. 
This identity expresses the \emph{energy} in terms of coordinates. As the matrix $Z$ varies measurably in $x$, it has been named a \emph{measurable Riemannian metric}, \cite{Hino13, Ka12, Ki08}. The following is version of (\ref{E:KuKi0})   immediately following from the chain rule \cite[Theorem 3.3.2]{FOT94}.

\begin{lemma}\label{L:KuKi}\mbox{}
Let $m$ be a smooth energy dominant measure and $(y^i)_{i\in I}$ a coordinate sequence. For all $f=F\circ y$ and $g=G\circ y$ from $\mathcal{P}(y)$ we have 
\begin{equation}\label{E:KuKi}
\Gamma(f,g)(x)=\left\langle \nabla F(y), Z(x)\nabla G(y)\right\rangle_{l_2}
\end{equation}
for $m$-a.e. $x\in X$. If in addition $f,g\in\mathcal{F}$, then
\[\mathcal{E}(f,g)=\int_X\left\langle \nabla F, Z\nabla G\right\rangle_{l_2}dm.\]
\end{lemma}

We rewrite (\ref{E:KuKi}) in a somewhat artificial way. For any $x\in X$ such that $Z(x)$ is symmetric and nonnegative definite, the bilinear extension of 
\begin{equation}\label{E:fiberSP}
\left\langle f_1\otimes g_1, f_2\otimes g_2\right\rangle_{\mathcal{H}_x}:=G_1(y)G_2(y)\left\langle \nabla F_1(y), Z(x)\nabla F_2(y)\right\rangle_{l_2},
\end{equation}
where $f_i=F_i\circ y$ and $g_i=G_i\circ y$ are members of $\mathcal{P}(y)$ with polynomials $F_i$ and $G_i$, $i=1,2$, defines a nonnegative definite symmetric bilinear form on the vector space $\mathcal{P}(y)\otimes\mathcal{P}(y)$. Let $\left\|\cdot\right\|_{\mathcal{H}_x}$ denote the associated Hilbert seminorm. Factoring out zero seminorm elements and completing, we obtain a Hilbert space $(\mathcal{H}_x, \left\langle\cdot,\cdot\right\rangle_{\mathcal{H},x})$. The $\mathcal{H}_x$-equivalence class of an element $f\otimes g$ of $\mathcal{P}(y)\otimes\mathcal{P}(y)$ we denote by $(f\otimes g)_x$. Note that for $m$-a.e. $x\in X$ the expression in (\ref{E:fiberSP}) equals
\[g_1(x)g_2(x)\frac{\Gamma(f_1,f_2)}{dm}(x).\]

\begin{examples}
\ (1) In the situation of Example \ref{Ex:1} (1) we observe  $\mathcal{H}_x\cong \mathbb{R}^n$ for $\lambda^n$-a.e. $x\in\mathbb{R}^n$ and
\[\left\langle f_1\otimes g_1, f_2\otimes g_2\right\rangle_{\mathcal{H}_x}=g_1(x)g_2(x)\left\langle \nabla f_1 (x), a(x)\nabla f_2(x)\right\rangle_{\mathbb{R}^n},\]
where we write $a=(a_{ij})_{i,j=1}^n$. 
\smallskip\item[\ (2)] For the Riemannian situation in Example \ref{Ex:1} (2) we have
\[\left\langle f_1\otimes g_1, f_2\otimes g_2\right\rangle_{\mathcal{H}_x}=g_1(x)g_2(x)\left\langle df_1(x), df_2(x)\right\rangle_{T_x^\ast M}\]
for $dvol$-a.e. $x\in U$, where 
\begin{equation}\label{E:extder}
f\mapsto df=\sum_{i=1}\frac{\partial f}{\partial y^i}dy^i
\end{equation}
denotes the exterior derivation. Note that $\mathcal{H}_x\cong T_x^\ast M\cong T_xM \cong \mathbb{R}^n$. 
\smallskip\item[\ (3)] For the Heisenberg group as in Example \ref{Ex:1} (3), 
\[\left\langle f_1\otimes g_1, f_2\otimes g_2\right\rangle_{\mathcal{H}_q}=g_1(q)g_2(q)\left((X(q)f_1)(X(q)f_2)+(Y(q)f_1)( Y(q)f_2)\right)\]
for $\lambda^3$-a.e. $q\in U$. Here $\mathcal{H}_q$ is isometrically isomorphic to the horizontal fiber $V_q$.
\end{examples}

We proceed to a more global perspective. A nonnegative definite symmetric bilinear form on $\mathcal{C}\otimes\mathcal{C}$ can be introduced by extending
\begin{equation}\label{E:HSP}
\left\langle f_1\otimes g_1, f_2\otimes g_2\right\rangle_{\mathcal{H}}:=\int_X g_1(x)g_2(x)\:\Gamma(f_1, f_2)(x)m(dx).
\end{equation}
The associated Hilbert seminorm is denoted by $\left\|\cdot\right\|_{\mathcal{H}}$. Factoring out zero seminorm elements and completing yields another Hilbert space $\mathcal{H}$, usually referred to a the \emph{Hilbert space of $1$-forms associated with $(\mathcal{E},\mathcal{F})$}. This definition has some history, see e.g. \cite{Eb99, N85, S90}, and in the context of Dirichlet forms it was first introduced by Cipriani and Sauvageot in \cite{CS03}. Right and left actions of $\mathcal{C}$ on the space $\mathcal{C}\otimes\mathcal{C}$ can be defined by extending 
\begin{equation}\label{E:actions}
(f\otimes g)h:=f\otimes (gh)\ \ \text{ and  }\ \ h(f\otimes g)=(fh)\otimes g - h\otimes (fg).
\end{equation}
By strong locality they coincide. Moreover, they extend further to an action of $\mathcal{C}$ on $\mathcal{H}$ and $\left\|\omega h\right\|_{\mathcal{H}}\leq \left\|h\right\|_{L_\infty(X,m)}\left\|\omega\right\|_{\mathcal{H}}$ for any $\omega\in\mathcal{H}$ and $h\in\mathcal{C}$. A linear operator $\partial:\mathcal{C}\to\mathcal{H}$ can be introduced by setting 
\[\partial f:=f\otimes \mathbf{1}, \ \ f\in\mathcal{C},\]
note that $f\otimes\mathbf{1}$ is a member of $\mathcal{H}$, as can be seen from (\ref{E:HSP}) by approximating $\mathbf{1}$ pointwise. The operator $\partial$ is a derivation, i.e. 
\begin{equation}\label{E:Leibniz}
\partial(fg)=(\partial f)g+f\partial g,\ \ f,g\in\mathcal{C}.
\end{equation}
It satisfies 
\begin{equation}\label{E:partialbound}
\left\|\partial f\right\|_{\mathcal{H}}^2=\mathcal{E}(f),\ \ f\in\mathcal{C},
\end{equation}
and extends to a closed unbounded operator $\partial: L_2(X,\mu)\to\mathcal{H}$ with domain $\mathcal{F}$. 

Since the left action in (\ref{E:actions}) is also well defined for bounded Borel functions, approximation shows that $(f\otimes g)\mathbf{1}_V=(\partial f)g\mathbf{1}_V$ is in $\mathcal{H}$ for any $f,g\in \mathcal{F}$ and relatively compact open $V$. By locality, (\ref{E:GammaV}) and approximation (pointwise $m$-a.e.) we then have $(f\otimes g)\mathbf{1}_V\in\mathcal{H}$ even for locally bounded $f,g\in\mathcal{F}_{loc}$. Formulas (\ref{E:actions}) and (\ref{E:Leibniz}) have local versions valid for elements of $\mathcal{P}(y)$. Note also that for $m$-a.e $x\in X$,
\[\left\langle(\partial f)_x, (\partial g)_x\right\rangle_{\mathcal{H}_x}=\frac{d\Gamma(f,g)}{dm}(x).\]

The next lemma contains a corresponding version of Lemma \ref{L:KuKi}.

\begin{lemma}\label{L:approx}
For $f_i=F_i\circ y$ and $g_i=G_i\circ y$ from $\mathcal{P}(y)$, $i=1,2$, and any relatively compact open $V$ we have 
\begin{align}
\left\langle (f_1\otimes g_1)\mathbf{1}_V, f_2\otimes g_2\right\rangle_{\mathcal{H}}&=\int_V\left\langle (f_1\otimes g_1)_x, (f_2\otimes g_2)_x\right\rangle_{\mathcal{H}_x}m(dx)\notag\\
&=\int_V G_1(y)G_2(y)\left\langle \nabla F_1(y), Z(x)\nabla F_2(y)\right\rangle_{l_2} m(dx).\notag
\end{align}
If in addition $f,g\in\mathcal{F}$, then we can replace $V$ by $X$. Moreover,
\[\lin\left(\left\lbrace (f\otimes g)\mathbf{1}_V  : f,g\in\mathcal{P}(y), \text{$V\subset X$ relatively compact open}\right\rbrace\right)\]
is a dense subspace of $\mathcal{H}$. If the coordinates $y^i$ have finite energy, then $\mathcal{P}(y)\otimes\mathcal{P}(y)$ is a dense subspace of $\mathcal{H}$.
\end{lemma}

\begin{proof}
The first statement is obvious. To see the second, let $\varphi$ and $\psi$ be functions from $\mathcal{C}$ and $U$ a relatively compact open set containing $\supp \varphi$ on which $\varphi$ is locally approximated  on $U$ by a sequence $(f_n)_n\subset \mathcal{P}(y)$. We have
$\left\|\varphi\otimes\psi-f_n\otimes\psi\right\|_{\mathcal{H}}^2\leq \sup_{x\in X}|\psi(x)|^2\Gamma(\varphi-f_n)(U)$, 
which converges to zero by Lemma \ref{L:locapprox}. Hence the span of elements 
$f\otimes \psi$ with $f\in\mathcal{P}(y)$ and $\psi\in\mathcal{C}$
is dense in $\mathcal{H}$. 
On the other hand, 
if $V$ is a relatively compact open set containing $\supp\psi$ and $(g_n)_n\subset\mathcal{P}(y)$ approximates $\psi$ locally on $V$, 
after replacing the sequence by a suitable subsequence
Lemma \ref{L:locapprox} implies $\left\|f\otimes \psi-(f\otimes g_n)\mathbf{1}_V\right\|_{\mathcal{H}}^2=\int_V(\psi-g_n)^2\Gamma(f)dm\to0$
by dominated convergence.
\end{proof}

\begin{examples}
\ (1) In  Example \ref{Ex:1}(1) the space $\mathcal{H}$  is isometrically isomorphic to the space $L_2(\mathbb{R}^n, \mathbb{R}^n)$ of $\mathbb{R}^n$-valued square integrable functions on~$\mathbb{R}^n$.
\smallskip\item[\ (2)] For the Riemannian situation in Example \ref{Ex:1}(2) the space $\mathcal{H}$ is isometrically isomorphic to the space $L_2(U, T^\ast M, dvol)$ of $L_2$-differential $1$-forms on $U\subset M$.
\end{examples}

\begin{remark}\label{R:locallyconvex}
\ (i) The spaces $\mathcal{H}_x$ may be seen as the \emph{fibers} of the measurable $L_2$-bundle $\mathcal{H}$. Formula (\ref{E:fiberSP}) expresses the \emph{fibers} in terms of coordinates.
\smallskip\item[\ (ii)]
The spaces $\mathcal{H}_x$ depend on the choice of $m$. However, the space $\mathcal{H}$ does not, as follows from (\ref{E:energymeas}) and (\ref{E:HSP}).
\smallskip\item[\ (ii)] If the coordinates $y^i$ have finite energy then we may replace $\mathcal{C}$ by $\mathcal{P}(y)$ in (\ref{E:HSP}) and the subsequent formulas. 
By Lemma \ref{L:approx}, regularity and \cite[Theorem 2.1.4]{FOT94} this yields the same space $\mathcal{H}$.
\smallskip\item[\ (iii)] We formulated (\ref{E:HSP}) and (\ref{E:actions}) in terms of the algebra $\mathcal{C}$ in order to use the same definition of the space of $1$-forms as in \cite{CS03,
HKT, HRT13, IRT}. Alternatively - and in view of Definition \ref{D:coordseq} this seems more appropriate - one can endow $\mathcal{P}(y)\otimes\mathcal{P}(y)$ with a directed family of Hilbert seminorms determined by $\left\|f\otimes g\right\|_{\mathcal{H}(V)}:=\left\|(f\otimes g)\mathbf{1}_V\right\|_{\mathcal{H}}$,
where the sets $V$ are relatively compact and open. This yields a presheaf of Hilbert spaces  whose inverse limit is a locally convex space $\mathcal{H}_{loc}$. Details can be found in \cite[Section 6]{HKT13}. Also $\mathcal{F}_{loc}$ may be viewed as a locally convex space, and the derivation $\partial$ may then be interpreted as a continuous linear operator from $\mathcal {F}_{loc}$ into $\mathcal{H}_{loc}$, if (\ref{E:partialbound}) is replaced by
$\left\|\partial f\right\|_{\mathcal{H}(V)}^2=\Gamma(f)(V)$, $f\in\mathcal{P}(V)$.
\end{remark}

\section{Differential and gradient in coordinates}

For any coordinate function $y^i$ and any relatively compact open set $V$ the element $(\partial y^i)\mathbf{1}_V$ is an element of $\mathcal{H}$. This implies the identities
\[\left\langle (\partial y^i)_x, (\partial y^j)_x\right\rangle_{\mathcal{H}_x}=Z^{ij}(x)\]
for $m$-a.e. $x\in X$. Moreover, the local version of (\ref{E:Leibniz}) shows that for any function $f=F(y^{n_1},\dots, y^{n_k})$ from $\mathcal{P}(y)$ we have on any locally compact open set $V$
\begin{equation}\label{E:chainrule}
\partial f= \sum_{i=1}^k\frac{\partial F}{\partial y^{n_i}}\partial y^{n_i}
\end{equation}

\begin{examples}
In the Euclidean and Riemannian situations (1) and (2) in Examples \ref{Ex:1} the operator $\partial$ may be identified with the exterior derivation and formula (\ref{E:chainrule}) becomes the classical identity in (\ref{E:extder}).
\end{examples}

The operator $\partial$ may be viewed as a \emph{generalization of the exterior derivation} and (\ref{E:chainrule}) may be viewed as a formula for the \emph{differential} $\partial f$ of $f$ in terms of coordinates.

On a general metric measure space a smooth theory of ordinary differential equations is not available. On the other hand the spaces $\mathcal{H}_x$ are Hilbert, hence self-dual. Therefore it seems artificial to rigorously distinguish between $1$-forms and vector fields. We interpret the elements of $\mathcal{H}$ also as \emph{(measurable) vector fields} and $\partial$ as a \emph{substitute for the gradient operator}.

Recall the notation in (\ref{E:union}). Given a finite ordered subset $J$ of $I$ let the collection of $\mathcal{H}_x$-equivalence classes of elements of $\mathcal{P}_J\otimes\mathcal{P}_J(y)$ be denoted by $\mathcal{H}_{x,J}$. Clearly this is a subspace of $\mathcal{H}_x$, and we have 
\[\mathcal{H}_x=\clos \left(\bigcup_{J\subset I} \mathcal{H}_{x,J}\right),\]
the union taken over all finite ordered subsets $J$ of $I$. 

Now suppose $J=(n_1,\dots, n_k)$. Formula (\ref{E:chainrule}) implies that the elements $(\partial y^{n_1})_x$, $\dots $, $(\partial y^{n_k})_x$ span $\mathcal{H}_{x,J}$. Let $Z_J(x)$ denote the matrix $(Z(x)^{n_i n_j})_{i,j=1}^k$, clearly symmetric and nonnegative definite. The preceding formulas yield another \emph{expression of the gradient $\partial f$, now in terms of the Euclidean gradient and the measurable metric $Z$}: For any $f=F\circ y \in \mathcal{P}_J(y)$ and any $j=1,\dots, k$ we have
\begin{equation}\label{E:gradientformula}
\left\langle (\partial f)_x, (\partial y^{n_j})_x\right\rangle_{\mathcal{H}_x}=\sum_{i=1}^k\frac{\partial F}{\partial y^{n_i}}(y)Z^{n_i n_j}(x)=\left(Z_J(x)\nabla F(y)\right)_j,
\end{equation}
where $\nabla F$ is the gradient of $F$ on $\mathbb{R}^k$.

\begin{examples}
\ (1) For Examples \ref{Ex:1} (1) we obtain
\[\left\langle (\partial f)_x, (\partial y^j)_x\right\rangle_{\mathcal{H}_x}=\sum_{i=1}^n a_{ij}(x)\frac{\partial f}{\partial y^i}(x)=\left(a(x)\nabla f(x)\right)_j.\]
\smallskip\item[\ (2)] In the Riemannian case of Examples \ref{Ex:1} (2) formula (\ref{E:gradientformula}) gives 
\[\left\langle (\partial f)_x, (\partial y^j)_x\right\rangle_{\mathcal{H}_x}=\left\langle df, dy^j\right\rangle_{T_x^\ast M}=\frac{\partial f}{\partial y^i}(x)\left\langle dy^i, dy^j\right\rangle_{T_x^\ast M}=g^{ij}(x)\frac{\partial f}{\partial y^i}(x).\]
This equals $dy^j(\grad f)$
because 
$\displaystyle \grad f= g^{ij}\frac{\partial f}{\partial y^i}\frac{\partial}{\partial y^j}$. 
\smallskip\item[\ (3)] Let $\left\langle\langle \cdot,\cdot \right\rangle\rangle$ denote the cometric associated with the Heisenberg group $\mathbb{H}$. Then
\[\left\langle (\partial f)_q, (\partial y^1)_q\right\rangle_{\mathcal{H}_q}=\sum_{i=1}^3 \frac{\partial f}{\partial y^i}(q)\left\langle\langle dy^i, dy^1\right\rangle\rangle=\left(Z(q)\nabla f(q)\right)_1=Xf(q).\]
In a similar manner we obtain
\[\left\langle (\partial f)_q, (\partial y^2)_q\right\rangle_{\mathcal{H}_q}=\left(Z(q)\nabla f(q)\right)_2=Yf(q)\]
\[\left\langle (\partial f)_q, (\partial y^3)_q\right\rangle_{\mathcal{H}_q}=\left(Z(q)\nabla f(q)\right)_3=-\frac{\eta}{2} Xf(q)+\frac{\xi}{2} Yf(q).\]
\end{examples}

\section{Divergence in coordinates}

By $-\partial^\ast$ we denote the adjoint of $\partial$, that is the unbounded linear operator $-\partial^\ast:\mathcal{H}\to L_2(X,\mu)$ with dense domain $\dom \partial^\ast$ and such that the integration by parts formula
\begin{equation}\label{E:ibp}
\left\langle v, \partial u\right\rangle_{\mathcal{H}}=-\left\langle \partial^\ast v,u\right\rangle_{L_2(X,\mu)}
\end{equation}
holds for all $v\in\dom \partial^\ast$ and $f\in\mathcal{F}$. We view the operator $-\partial^\ast$ both ways, as \emph{coderivation} and as \emph{divergence operator}.

In the context of coordinates it is more suitable to deviate a bit from the Hilbert space interpretation in (\ref{E:ibp}). First assume that all coordinates  $y^i$ have finite energy. For an element $(\partial f)g\mathbf{1}_V$ of $\mathcal{H}$ with $f,g\in\mathcal{P}(y)$ we then set
\[\partial^\ast((\partial f)g)(u):=-\left\langle (\partial f)g, (\partial u)\right\rangle_{\mathcal{H}},\ \ u\in\mathcal{P}(y).\]
By Cauchy-Schwarz 
$|\partial^\ast((\partial f)g)(u)|\leq \left\|(\partial f)g\right\|_{\mathcal{H}}\:\mathcal{E}(u)$, and therefore $\partial^\ast(\partial f)g$ may be seen as a continuous linear functional on $\mathcal{P}(u)$, and after a straighforward extension by Definition \ref{D:coordseq} and regularity, on $\mathcal{F}$.

As before let $J=(n_1,\dots, n_k)$. Given functions polynomials $F$ and $G$ in $y^{n_1}$, $\dots$, $y^{n_k}$ and a function $u=U\circ y$ with $U\in C^1(\mathbb{R}^k)$ put
\[\diverg_{Z_J}(G\nabla F)(U):=-\sum_{i,j=1}^k \int_X G(y)\frac{\partial F}{\partial y^{n_i}}(y)Z^{n_i,n_j}(x)\frac{\partial U}{\partial y^{n_j}}(y)m(dx).\] 
Then
\begin{equation}\label{E:divergform}
\partial^\ast((\partial f)g)(u)=\diverg_{Z_J}(G\nabla F)(U)
\end{equation}
provides a 'distributional' \emph{coordinate expression for the divergence}. Of course this is a naive definition by duality, and in particular we have $\partial^\ast((\partial f))(u)=-\mathcal{E}(f,u)$. In general there is no integration by parts formula on the level of coordinates that could permit a more interesting definition.

If the coordinates $y^i$ do not have finite energy, we view
$\mathcal{P}(y)$ as a locally convex space, then $\partial^\ast(\partial f)g\mathbf{1}_V$ with  relatively compact open $V$ defines a continuous linear functional on $\mathcal{P}(y)$. Proceding similarly as before one obtains local versions of (\ref{E:divergform}).
\begin{examples}
\ (1) For Example \ref{Ex:1} (1) we obtain 
\[\diverg_a(g\nabla f)(u)= -\sum_{i,j=1}^n\int_{\mathbb{R}^n}g(x)\frac{\partial f}{\partial x_i}(x) a_{ij}(x)\frac{\partial u}{\partial x_j}(x)dx\]
for any $u\in C_c^1(\mathbb{R}^n)$. If in addition the coefficients $a_{ij}$ are $C^1$, this is seen to equal
\[\int_{\mathbb{R}^n}\diverg(a(g\nabla f)) u \:dx.\]
\smallskip\item[\ (2)] In the Riemannian situation of Examples \ref{Ex:1} (ii) we have 
\[\diverg_g(h\nabla f)(u)=\int_{W} g^{ij}h\frac{\partial f}{\partial y^i}\frac{\partial u}{\partial y^j}\sqrt{g} dy^1\cdots dy^n=\int_U \diverg(h\grad f)\:u\:dvol\]
for any $u\in C_c^1(U)$, where $g:=\det (g_{ij})$ and 
\[\diverg(h\grad f)=\frac{1}{\sqrt{g}}\frac{\partial}{\partial y^j}\left(\sqrt{g} g^{ij}h\frac{\partial f}{\partial y^i}\right)\] 
is the divergence of $h\grad f$ in the usual Riemannian sense. See \cite[Section 2.1]{Jost}.
\smallskip\item[\ (3)] In Example \ref{Ex:1} (3) formula (\ref{E:divergform}) yields
\[\diverg_Z(g\nabla f)(u)=\sum_{i,j=1}^3\int_U Z^{ij} g\frac{\partial f}{\partial y^i}\frac{\partial u}{\partial y^j}\:d\lambda^3=\sum_{i,j=1}^3\int_U\frac{\partial}{\partial y^j}\left(Z^{ij} g\frac{\partial f}{\partial y^i}\right)\:u\:d\lambda^3\]
for any $u\in C_c^1(U)$, what equals
\[\int_U\left(\frac{\partial}{\partial \xi}gXf+\frac{\partial}{\partial\eta}gYf+\frac{\partial}{\partial\zeta}\left(g(-\frac{\eta}{2}Xf+\frac{\xi}{2}Yf)\right)\right)\:u\:d\lambda^3=\int_U\diverg(Z(g\nabla f))\:u\:d\lambda^3,\]
where $\diverg$ is the ordinary divergence operator on $\mathbb{R}^3$.
\end{examples}

\section{Generator in coordinates}

We consider the infinitesimal generator $(L, \dom\:L)$ of $(\mathcal{E},\mathcal{F})$. From (\ref{E:ibp}) and the definition of the adjoint we see that for any $f\in \dom \:L$ we have $\partial f\in\dom\:\partial^\ast$ and 
\begin{equation}\label{E:divgrad}
Lf=\partial^\ast\partial f.
\end{equation}
Although in general a coordinate version of this formula may not be available,
it can be written in terms of coordinates for specific examples. 

To express $L$ in coordinates additional assumptions are inevitable. Even if $y^i\in \dom L$ for all $i$ the inclusion $\mathcal{P}(y)\subset \dom L$ holds if and only if the reference measure $\mu$ itself is energy dominant, that is if $(\mathcal{E},\mathcal{F})$ admits a \emph{carr\'e du champ} in the sense of \cite{BH91}.  For Examples \ref{Ex:1} (1)-(4) this is satisfied. However, the standard resistance form on the Sierpinski gasket, considered as a Dirichlet form with respect to the natural self-similar Hausdorff measure, does not have this property, and this situation is typical for a large class of self-similar spaces, \cite{BBST, Hino05, HinoNak06}.
\begin{assumption}\label{A:carre}
\text{The reference measure $\mu$ itself is energy dominant.}
\end{assumption}
Let $(L, \dom_{(1)} L)$ denote the smallest closed extension of the restriction of $L$ to 
$$\left\lbrace f\in \dom L\cap L_1(X,\mu): Lf\in L_1(X,\mu)\right\rbrace.$$ Assumption \ref{A:carre} is known to be necessary and sufficient for $\dom_{(1)} L \cap L_\infty(X,\mu)$ to be an algebra under pointwise multiplication. If it is in force, then $f,g\in \dom L$ implies $fg\in\dom_{(1)} L$ and we have
\begin{equation}\label{E:carreformula}
\frac{d\Gamma(f,g)}{d\mu}=L(fg)-fLg-gLf, 
\end{equation}
see \cite[Theorems I.4.2.1 and I.4.2.2]{BH91}. To formulate local conditions on the coordinate functions we follow \cite[Definition 4.2 (2)]{LT13} and say that a function $f\in L_{2,loc}(X,\mu)$ \emph{belongs to the strong local domain $\dom_{loc} L$ of $L$} if for any relatively compact open set $V$ there exists some $u\in\mathcal{F}$ such that $f|_V=u|_V$ $\mu$-a.e. Similarly we define $\dom_{(1),\:loc} L$. Then identity (\ref{E:carreformula}) holds for any $f,g\in \dom_{loc}\:L$ locally on any relatively compact open set $V$.
\begin{assumption}\label{A:domloc}
\text{The coordinates $y^i$ are members of $\dom_{loc}\:L$.}
\end{assumption} 

Let Assumptions \ref{A:carre} and \ref{A:domloc} be in force. This implies $\mathcal{P}(y)\subset \dom_{loc}\:L$. Suppose $f=F\circ y\in\mathcal{P}(y)$, where again $J=(n_1,\dots, n_k)$. Using (\ref{E:carreformula}) on the coordinates $y^i$ and iterating, we inductively arrive at a \emph{coordinate formula for the generator}
\[Lf(x)=\sum_{i,j=1}^k \frac{\partial^2 F}{\partial y^{n_j}\partial y^{n_j}}(y)Z^{n_i n_j}(x)+\sum_{i=1}^k\frac{\partial F}{\partial y^{n_i}}(y)Ly^{n_i}(x),\]
valid locally on any relatively compact open $V$. This is a version of a well known identity, see e.g. \cite[Lemma 6.1]{Em89} or \cite{Eb99}.

\begin{examples}
\ (1) For Example \ref{Ex:1} (1) with $C^1$-coefficients $a_{ij}$ we have 
\[Lf=\diverg(a\nabla f)=\sum_{i,j=1}^n\frac{\partial^2 f}{\partial x_i\partial x_j}a_{ij}+\sum_{i=1}^n \frac{\partial f}{\partial x_i}\sum_{j=1}^n\frac{\partial a_{ij}}{\partial x_j}.\]
\smallskip\item[\ (2)] For Example \ref{Ex:1} (2) we observe
\[\Delta f=\diverg(\grad f)=\frac{1}{\sqrt{g}}\frac{\partial}{\partial y^j}\left(\sqrt{g} g^{ij}\frac{\partial f}{\partial y^i}\right),\]
what differs by a minus sign from the Laplace-Beltrami operator (convention). 
\smallskip\item[\ (3)] For Example \ref{Ex:1} (3) arrive at the Heisenberg sub-Laplacian,
\begin{align}
Lf=\diverg(Z\nabla f)&=\frac{\partial^2 f}{\partial\xi^2}+\frac{\partial^2 f}{\partial \eta^2}+\xi\frac{\partial^2 f}{\partial\eta\partial\zeta}+\eta\frac{\partial^2 f}{\partial\xi\partial\zeta}+\frac{\xi^2+\eta^2}{4}\frac{\partial^2 f}{\partial\zeta^2}\notag
=\left( X^2+Y^2\right)f.\notag
\end{align}
\smallskip\item[\ (4)] In Example \ref{Ex:1} (4) the Dirichlet form generator of $(\mathcal{E},\mathcal{F})$ on $L_2(K,\nu)$ is the Kusuoka Laplacian $(\Delta_\nu, \dom\:\Delta_\nu)$. The coordinate functions $y^i$ are harmonic, that is $y^i \in \dom\:\Delta_{\nu}$ and $\Delta_\nu y^i=0$, $i=1,2$. Accordingly we have
\[\Delta_\nu f(x)=\sum_{i,j=1}^2 \frac{\partial^2 F}{\partial y^i\partial y^j}(y) Z^{ij}(x)\]
for any $f=F\circ y\in\mathcal{P}(y)$. This can be rewritten as $\tr(Z(x) D^2F(y))$, where $D^2F$ is the Hessian of $F$ and $\tr$ the trace operator, see \cite[Theorem 8]{T08}. 
\end{examples}

\section{Constructing coordinate sequences}\label{S:construct}

Let $(\mathcal{E},\mathcal{F}$ be a strongly local regular Dirichlet form. Under some continuity condition it is always possible to simultaneously construct a smooth energy dominant measure and a corresponding coordinate sequence. The latter may be designed to have nice decay properties. Let $(P_t)_{t>0}$ denote the Markovian semigroup uniquely associated with $(\mathcal{E}, \mathcal{F})$, \cite{ChFu12, FOT94}. If it is also  a strongly continuous semigroup of contractions $P_t: C_0(X)\to C_0(X)$ on the space $C_0(X)$ of continuous functions vanishing at infinity, then it is called a \emph{Feller semigroup}.

\begin{examples}
The transition semigroups of many diffusion processes of Euclidean domains or manifolds are Feller semigroups. Also the semigroups of many diffusions on fractals are known to be Feller, see for instance \cite{Ba98, BB99, BBKT, Ki01}.
\end{examples}

\begin{lemma}\label{L:construct}
Assume that the semigroup $(P_t)_{t>0}$ is a Feller semigroup. Then there exists a finite smooth energy dominant measure $\widetilde{m}$ and a coordinate sequence $(y^i)_{i\in I}\subset \dom L$ for $(\mathcal{E},\mathcal{F})$ with respect to $\widetilde{m}$ such that 
\smallskip\item[\ (i)] $\lin (\left\lbrace y^i\right\rbrace_{i\in I})$ is dense in $\mathcal{F}$,
\smallskip\item[\ (ii)] for any $i$ also the functions $Ly^i$ are continuous,
\smallskip\item[\ (iii)] we have
$\displaystyle \sum_{i=1}^\infty\left\|y^i\right\|_{\sup}^2<+\infty\ \ \text{ and }\ \ \sum_{i=1}^\infty \left\|Ly^i\right\|_{\sup}^2<+\infty.$
\end{lemma}


\begin{proof}
Let $\left\lbrace f_i\right\rbrace_i\subset C_c(X)$ be a countable family of nonzero functions that is dense in $L_2(X,\mu)$. By the Feller property,  the resolvent functions
$ G_1f_i(x):=\int_0^\infty e^{-t}P_tf(x)dt$, 
are continuous and  $G_1f_i\in \dom L$. Set

\centerline{$ y^i:=
{2^{-n}G_1f_i}\Big/{\left(\left\|G_1 f_i\right\|_{\sup}+\left\|f_i\right\|_{\sup}+\mathcal{E}(G_1f_i)^{1/2}\right)}.$}

\noindent
Then (ii) and (iii) are satisfied. The range $\im G_1$ of $G_1:L_2(X,\mu)\to L_2(X,\mu)$ is dense in $\mathcal{F}$ and any element of $\im G_1$ can be approximated in $\mathcal{F}$ by linear combinations of the functions $G_1f_i$, what implies (i). Now set $ \widetilde{m}:=\sum_{i=1}^\infty 2^i \Gamma(y^i)$. 
Because the energy measures $\Gamma(y^i)$ satisfy
$\displaystyle \Gamma(y^i)\leq \frac{\Gamma(G_1 f_i)}{2^{2n}\mathcal{E}(G_1f_i)}\leq 2^{-2i}$, 
we have $\widetilde{m}(X)\leq \sum_{i=1}^\infty 2^{-i}<+\infty$. Since all energy measures are smooth, so is $\widetilde{m}$. For the densities 
we observe
$ Z^{ii}=\frac{d\Gamma(y^i)}{d\widetilde{m}}\leq \frac{d\Gamma(y^i)}{2^i d\Gamma(y^i)}\leq 2^{-i}\ \ \text{ $\widetilde{m}$-a.e.}$
Polarizing and choosing appropriate $\widetilde{m}$-versions of the functions $Z^{ij}$, we may assume that for $m$-a.e. $x\in X$ and any $N\in\mathbb{N}$ the matrix $(Z^{ij}(x))_{i,j=1}^N$ is symmetric and nonnegative definite. To do so it suffices to note that given $v_1,\dots, v_N\in\mathbb{R}$,
$ 0\leq \Gamma\left(\sum_{i=1}^N v_iy^i\right)(A)=\int_A\sum_{i=1}^N Z^{ij}(x)v_iv_j\:\widetilde{m}(dx)$
is a nonnegative Radon measure, hence its density must be nonnegative $\widetilde{m}$-a.e. By letting $N$ go to infinity we can finally obtain  

\centerline{$\displaystyle
\left\|Z(x)v\right\|_{l_2}^2\leq 
\sum_{i,j} |Z^{ij}(x)|^2|v_j|^2\leq 
\sum_{i,j} |Z^{ii}(x)||Z^{jj}(x)||v_j|^2\leq 
\sum_{i,j} 2^{-i-j}|v_j|^2\leq \left\|v\right\|_{l_2}^2
$}

\noindent
for any $v=(v_1,v_2,\dots)\in l_2$, what allows to conclude that $Z(x)$ is bounded, symmetric and nonnegative definite on $l_2$ for $\mu$-a.e. $x\in X$.
\end{proof}

\end{document}